\documentclass{article}
\usepackage{college-math-j}

\theoremstyle{theorem}
\newtheorem{theorem}{Theorem}
\newtheorem{proposition}{Proposition}
\newtheorem{corollary}{Corollary}
\theoremstyle{definition}
\newtheorem*{remark}{Remark}

\begin{document}

\title{From Pythagorean Runs to Cubic Runs}
\author{Anatoly Eydelzon\\
\small Department of Mathematical Sciences, The University of Texas at Dallas\\
\small Richardson, Texas, USA\\
\small \texttt{anatoly@utdallas.edu}}
\date{}
\maketitle

\begin{abstract}
Pythagorean runs give a simple way to balance consecutive squares, with a formula that works for every length. Replacing squares by cubes changes the picture completely. A computer search turns up identities in which consecutive cubes are balanced by later cubes whose bases differ by two. After centering both progressions and fixing the ratio of their centers, the equation becomes a generalized Pell equation. One concrete seed then produces an infinite family of ordered cubic runs. The story therefore links a very elementary question about sums of cubes to Pell equations and linear recurrences.
\end{abstract}

\section{Introduction}

Several years ago, at a
retirement party of my esteemed colleague Dr. Paul Stanford, somebody mentioned the identity
\[
  3^2+4^2=5^2
\]
and the next identities of the same kind, and asked whether this pattern
could be continued in general.  I found the formula quickly enough, only to discover that Boardman had already published it as “Pythagorean Runs.” I set the problem aside for several years.

A few years later I returned to the same idea, but with cubes instead of squares. The change looks minor; the mathematics is not. The quadratic problem has a simple formula for every length.
For cubes, the search leads instead to Diophantine equations, Pell
equations, and some surprisingly large solutions.  This is the problem
considered in this paper.

Boardman's ``Pythagorean Runs'' gives the identities
\[
 3^2+4^2=5^2,
\]
\[
 10^2+11^2+12^2=13^2+14^2,
\]
\[
 21^2+22^2+23^2+24^2=25^2+26^2+27^2,
\]
and their general continuation \cite{Boardman}.  In modern notation,
for every \(K\ge1\), if
\[
 A=K(2K+1),
\]
then
\begin{equation}\label{eq:square-run}
 \sum_{j=0}^{K}(A+j)^2
 =
 \sum_{j=K+1}^{2K}(A+j)^2.
\end{equation}
Indeed, direct simplification gives
\[
 \sum_{j=0}^{K}(A+j)^2-
 \sum_{j=K+1}^{2K}(A+j)^2
 =
 (A+K)\bigl(A-K(2K+1)\bigr).
\]
Thus \(A=K(2K+1)\) is the unique positive solution for each \(K\).

It is natural to ask the same question for higher powers. I first tried the most direct cubic version, with consecutive cubes on both sides (so \(d=1\)). A computer search did not find an example. I then allowed the terms on the right to form an arithmetic progression. The first nontrivial case I found was \(d=2\), and this case leads to a Pell-type equation. This does not prove that no solutions with \(d=1\) exist; it only means that none were found in the search.

A first indication that the cubic problem may look different from Boardman's square runs is the identity

\begin{equation}\label{eq:smallintro}
 25^3+26^3+\cdots+32^3
 =
 33^3+35^3+37^3+39^3
 =188784.
\end{equation}
Here eight consecutive cubes are balanced by four later cubes whose
bases have common difference \(2\).  Identity \eqref{eq:smallintro}
can simply be checked by evaluating the two sums.  It was found by a
finite search rather than guessed.  Choose small positive integers
\(m\) and \(A\), and compute exactly
\[
 L(m,A)=\sum_{j=0}^{2m-1}(A+j)^3.
\]
For this fixed \(m\), increase \(B\) and compute
\[
 R(m,B)=\sum_{j=0}^{m-1}(B+2j)^3.
\]
Since \(R(m,B)\) is strictly increasing in \(B\), once
\(R(m,B)>L(m,A)\) there is no reason to try larger values of \(B\)
for that pair \(m,A\).  Testing small integers in this way produces
\[
 (m,A,B)=(4,25,33),
\]
and direct substitution gives \eqref{eq:smallintro}.  The small example was found by search, not guessed. It suggested the problem we now state. (The infinite family proved later does not rely on that search.)

This suggests the cubic problem
\begin{equation}\label{eq:main}
 \sum_{j=0}^{2m-1}(A+j)^3
 =
 \sum_{j=0}^{m-1}(B+2j)^3,
\end{equation}
together with the ordering condition
\begin{equation}\label{eq:ordering}
 B>A+2m-1.
\end{equation}
Thus the left side consists of \(2m\) consecutive cubes, while the
right side consists of \(m\) later cubes with common difference \(2\).
The main result of this note is that such ordered identities occur
infinitely often.

Problems involving sums of consecutive cubes and equal power sums of
arithmetic progressions have been studied in several settings; see, for
example, \cite{BazsoEtAl2012,BazsoMezo2015,BennettPatelSiksek,Pletser}.
Pell equations also occur in related problems involving consecutive cubes
\cite{Pletser}.  Our goal is narrower and more elementary: to show how the
particular step-\(1\) versus step-\(2\) problem \eqref{eq:main} leads from
a searched example to an explicit infinite family that also preserves the
strict ordering \eqref{eq:ordering}.

\section{A centered form of the cubic equation}

The special lengths \(2m\) and \(m\) make the equation particularly
simple after centering the two progressions.  Put
\begin{equation}\label{eq:XY}
 X=2A+2m-1,\qquad Y=B+m-1.
\end{equation}
Indeed, the center of the left-hand block is
\[
 \frac{A+(A+2m-1)}{2}=\frac{X}{2},
\]
while the center of the right-hand progression is
\[
 \frac{B+(B+2m-2)}{2}=Y.
\]
Thus \(X/2\) and \(Y\) are respectively the centers of the two
progressions.  These formulas just come from adding the first and last terms of each
progression and dividing by \(2\).  Conversely, once \(m,X,Y\) are known, the original starting
points are recovered by solving \eqref{eq:XY}:
\begin{equation}\label{eq:recoverAB}
 A=\frac{X-2m+1}{2},\qquad B=Y-m+1.
\end{equation}
This change of variables is useful because the formulas are simpler when
we use the centers of the two progressions.

\begin{proposition}\label{prop:centered}
For positive integers \(m,A,B\), equation \eqref{eq:main} is equivalent
to
\begin{equation}\label{eq:centered}
 X\bigl(X^2+4m^2-1\bigr)
 =
 4Y\bigl(Y^2+m^2-1\bigr),
\end{equation}
where \(X,Y\) are defined by \eqref{eq:XY}.
\end{proposition}

\begin{proof}
Using
\[
 \sum_{j=0}^{n-1}j=\frac{n(n-1)}2,\qquad
 \sum_{j=0}^{n-1}j^2=\frac{n(n-1)(2n-1)}6,\qquad
 \sum_{j=0}^{n-1}j^3=\left(\frac{n(n-1)}2\right)^2,
\]
expand both sides of \eqref{eq:main}.  Substituting
\[
 A=\frac{X-2m+1}{2},\qquad B=Y-m+1
\]
and collecting terms gives
\[
 \sum_{j=0}^{2m-1}(A+j)^3
 -
 \sum_{j=0}^{m-1}(B+2j)^3
 =
 \frac{m}{4}
 \left(
 X^3+4Xm^2-X-4Y^3-4Ym^2+4Y
 \right).
\]
Since \(m>0\), equality of the two cubic sums is therefore equivalent
to
\[
 X^3+4Xm^2-X-4Y^3-4Ym^2+4Y=0,
\]
which is exactly \eqref{eq:centered}.
\end{proof}

The identity \eqref{eq:smallintro} corresponds to
\begin{equation}\label{eq:small}
 (m,A,B)=(4,25,33).
\end{equation}
For this solution,
\[
 (X,Y)=(57,36).
\]
The infinite ordered family below will be generated from a different
seed.

\section{Reduction to a generalized Pell equation}

A second solution is
\begin{align}\label{eq:seedidentity}
705^3+706^3+\cdots+880^3
&=913^3+915^3+\cdots+1087^3 \notag\\
&=88,681,384,000.
\end{align}
This identity was found by the same finite computer search described above.

For this solution,
\[
 (X,Y)=(1585,1000),
\]
so that
\[
 \frac{X}{Y}=\frac{317}{200}.
\]
The ratio \(317:200\) was not chosen in advance.  It appeared when the
searched solution was written in centered coordinates.  This led to a natural
question: if we keep this center ratio fixed, can this one solution produce
infinitely many more?
Writing a common scale factor as \(z\) gives
\begin{equation}\label{eq:ratio}
 X=317z,\qquad Y=200z.
\end{equation}
For the displayed solution, \(z=5\).  With the ratio fixed, the three-variable cubic equation becomes a quadratic
equation in \(z\) and \(m\).  This is where the Pell equation enters.

\begin{proposition}\label{prop:pellreduction}
Under \eqref{eq:ratio}, the centered equation \eqref{eq:centered} is
equivalent to
\begin{equation}\label{eq:generalizedpell}
 48329z^2-156m^2=161.
\end{equation}
Moreover, \((z,m)=(5,88)\) is a positive integer solution.
\end{proposition}

\begin{proof}
Substitute \(X=317z\) and \(Y=200z\) into
\eqref{eq:centered}.  The two sides become
\[
 317z\bigl(317^2z^2+4m^2-1\bigr)
 \quad\hbox{and}\quad
 800z\bigl(200^2z^2+m^2-1\bigr).
\]
Since \(z>0\), divide by \(z\), move all terms to one side, and
collect the coefficients of \(z^2\), \(m^2\), and the constant term.
This gives
\[
 -144987z^2+468m^2+483=0.
\]
Dividing by \(3\) and rearranging gives
\[
 48329z^2-156m^2=161,
\]
which is \eqref{eq:generalizedpell}.  Finally,
\[
 48329\cdot5^2-156\cdot88^2=161.
\]
\end{proof}

Equation \eqref{eq:generalizedpell} is a Pell-type quadratic equation rather
than a generalized Pell equation in normalized form.  Multiplying it by
$48329$ and setting
\[
 x=48329z
\]
gives
\begin{equation}\label{eq:normalizedpell}
 x^2-7539324m^2=7780969,
\end{equation}
which has the standard generalized Pell form $x^2-Dm^2=N$, with
\[
D=7539324=48329\cdot156.
\]
To generate new solutions of this generalized equation, we use the
associated ordinary Pell equation
\begin{equation}\label{eq:pell}
 u^2-Dv^2=1,
\end{equation}
where \(u\) and \(v\) are positive integers.  These are new variables,
distinct from \(x\) and \(m\).  The role of a solution \((u,v)\) is to
provide a transformation that carries one solution of the generalized
Pell equation to another.

The numbers \(u\) and \(v\) are obtained from the continued fraction
of \(\sqrt D\).  Since \(D\) is a positive integer that is not a perfect
square, the simple continued fraction of \(\sqrt D\) is periodic:
\[
\sqrt D=[a_0;\overline{a_1,a_2,\ldots,a_r}].
\]
Its convergents \(p_k/q_k\) are computed from
\[
p_{-2}=0,\quad p_{-1}=1,\qquad
q_{-2}=1,\quad q_{-1}=0,
\]
and
\[
p_k=a_kp_{k-1}+p_{k-2},\qquad
q_k=a_kq_{k-1}+q_{k-2}.
\]
The first convergent for which
\[
p_k^2-Dq_k^2=1
\]
gives the least positive solution of the ordinary Pell equation:
\[
u=p_k,\qquad v=q_k.
\]
For \(D=7539324\), this procedure gives
\begin{equation}\label{eq:fundamental}
 u=2047435741791027199,\qquad
 v=745665546122880.
\end{equation}
These integers are very large, but they come directly from the continued
fraction of \(\sqrt D\); they are not obtained by a brute-force search.

There is a simple reason why \(\sqrt D\) is useful here.  If
\[
 x^2-Dm^2=N,
\]
then
\[
 (x+m\sqrt D)(x-m\sqrt D)=N.
\]
Similarly, equation \eqref{eq:pell} gives
\[
 (u+v\sqrt D)(u-v\sqrt D)=1.
\]
Consequently, multiplying a solution \(x+m\sqrt D\) of the generalized
Pell equation by \(u+v\sqrt D\) preserves the value \(N\):
\[
 (x+m\sqrt D)(u+v\sqrt D)=x'+m'\sqrt D,
\]
where
\begin{equation}\label{eq:pellmultiply}
 x'=ux+Dvm,\qquad m'=vx+um.
\end{equation}
Thus a solution \((u,v)\) of the ordinary Pell equation provides the
mechanism for generating new solutions of the generalized equation.

I will use two simple terms.  A \emph{seed} is the particular integer solution
where we start.  The \emph{Pell orbit of the seed} is the sequence obtained
by repeatedly multiplying by the same solution \(u+v\sqrt D\).  In symbols,
\[
 x_n+m_n\sqrt D
 =(x_0+m_0\sqrt D)(u+v\sqrt D)^n,
 \qquad n=0,1,2,\ldots.
\]
Here ``orbit'' only means the successive points obtained by repeating this
same transformation.

In our equation \eqref{eq:normalizedpell}, the seed
\((z_0,m_0)=(5,88)\) comes from the searched cubic run: indeed,
\(X=1585=317\cdot5\), so \(z_0=5\).  Since the normalized variable
was defined by \(x=48329z\), the corresponding first coordinate is
\[
 x_0=48329z_0=48329\cdot5=241645.
\]
Thus the seed for the normalized equation is
\[
 (x_0,m_0)=(241645,88).
\]
Returning from \(x\) to \(z\) in \eqref{eq:pellmultiply} gives the
recurrence in the next section.

\section{An infinite family of ordered cubic runs}

Set
\[
 a=48329,\qquad b=156,\qquad D=ab.
\]
Let \(u,v\) be as in \eqref{eq:fundamental}.  Starting from the
seed
\[
 (z_0,m_0)=(5,88),
\]
define
\begin{equation}\label{eq:recurrence}
 \begin{aligned}
 z_{n+1}&=u z_n+bv\,m_n,\\
 m_{n+1}&=av\,z_n+u m_n.
 \end{aligned}
\end{equation}
This recurrence is simply the Pell multiplication rule in the variables
\(z,m\).  Indeed, since \(x=az\) and \(D=ab\), formula
\eqref{eq:pellmultiply} gives
\[
 az'=u(az)+abv m,
 \qquad
 m'=v(az)+um,
\]
and division of the first equation by \(a\) yields
\[
 z'=uz+bvm,\qquad m'=avz+um.
\]
Thus \eqref{eq:recurrence} is simply the Pell orbit of the seed, written
back in the original variables.

The construction is now straightforward.  The seed gives the center ratio
\(317:200\).  Keeping this ratio changes the centered cubic equation into
\eqref{eq:generalizedpell}, and a Pell unit gives infinitely many pairs
\((z_n,m_n)\).  For each pair we put \(X_n=317z_n\), \(Y_n=200z_n\),
and recover \(A_n,B_n\) from \eqref{eq:recoverAB}.  We still have to
check that these are positive integers, that the right-hand block starts
after the left-hand block ends, and that the solutions are all different.

\begin{theorem}\label{thm:main}
For every \(n\ge0\), define
\begin{equation}\label{eq:AB}
 A_n=\frac{317z_n-2m_n+1}{2},
 \qquad
 B_n=200z_n-m_n+1,
\end{equation}
where \((z_n,m_n)\) is generated by \eqref{eq:recurrence}.  Then
\(m_n,A_n,B_n\) are positive integers,
\[
 B_n>A_n+2m_n-1,
\]
and
\begin{equation}\label{eq:family}
 \sum_{j=0}^{2m_n-1}(A_n+j)^3
 =
 \sum_{j=0}^{m_n-1}(B_n+2j)^3.
\end{equation}
Consequently there are infinitely many ordered cubic runs of the form
\eqref{eq:main}.
\end{theorem}

\begin{proof}
The matrix
\[
 M=
 \begin{pmatrix}
 u&bv\\
 av&u
 \end{pmatrix}
\]
has determinant \(u^2-abv^2=1\).  A direct calculation shows that
\[
 a z_{n+1}^2-bm_{n+1}^2
 =
 a z_n^2-bm_n^2.
\]
Hence every pair generated from \((5,88)\) satisfies
\[
 az_n^2-bm_n^2=161.
\]
All entries in the recurrence are positive, so \(z_n,m_n>0\).

The integer \(u\) is odd and \(v\) is even.  Since \(z_0\) is odd and
\(m_0\) is even, recurrence \eqref{eq:recurrence} preserves
\(z_n\) odd and \(m_n\) even.  Therefore \(A_n\) in
\eqref{eq:AB} is an integer, and \(B_n\) is obviously an integer.

Since
\[
 az_n^2-bm_n^2=161>0,
\]
we have
\[
 \frac{m_n}{z_n}<\sqrt{\frac{a}{b}}
 =\sqrt{\frac{48329}{156}}<\frac{83}{4}.
\]
It follows that
\[
 A_n
 =\frac{317z_n-2m_n+1}{2}>0,
 \qquad
 B_n=200z_n-m_n+1>0.
\]
Furthermore,
\begin{align*}
 B_n-(A_n+2m_n-1)
 &=\frac{83z_n-4m_n+3}{2}\\
 &>0.
\end{align*}
Thus the right-hand progression begins strictly after the left-hand
consecutive block ends.

Finally, put
\[
 X_n=317z_n,\qquad Y_n=200z_n.
\]
Proposition \ref{prop:pellreduction} gives
\eqref{eq:centered}; Proposition \ref{prop:centered} then gives
\eqref{eq:family}.  The recurrence is strictly increasing, hence
produces infinitely many distinct triples.
\end{proof}

\begin{corollary}[A scalar recurrence for the lengths]\label{cor:scalarrecurrence}
The sequence of half-lengths \(m_n\) in Theorem \ref{thm:main} satisfies
\[
 m_0=88,\qquad
 m_1=360360696170473731112,
\]
and, for every \(n\ge0\),
\begin{equation}\label{eq:scalarrecurrence}
 m_{n+2}=4094871483582054398\,m_{n+1}-m_n.
\end{equation}
\begin{proof}
The recurrence \eqref{eq:recurrence} is multiplication by the matrix
\[
 M=\begin{pmatrix}u&bv\\ av&u\end{pmatrix}.
\]
Since \(\det M=1\) and \(\operatorname{tr}M=2u\), the
Cayley--Hamilton identity gives
\[
 M^2-2uM+I=0.
\]
Applying this identity to \((z_n,m_n)^T\) shows that each coordinate
satisfies \(s_{n+2}=2u s_{n+1}-s_n\).  Using
\(u=2047435741791027199\) gives \eqref{eq:scalarrecurrence}.  Finally,
substitution of \((z_0,m_0)=(5,88)\) in \eqref{eq:recurrence} gives the
displayed value of \(m_1\).
\end{proof}
\end{corollary}

\begin{remark}
For example, the first four terms of the half-length sequence can be
generated in MATLAB as follows:
\begin{verbatim}
a = zeros(1,4,'sym');
a(1) = sym(88);
a(2) = sym('360360696170473731112');
for n = 1:2
    a(n+2) = sym('4094871483582054398')*a(n+1) - a(n);
end
a
\end{verbatim}
\end{remark}

\begin{remark}
The initial member \(n=0\) of Theorem \ref{thm:main} is the seed

\[
(m_0,A_0,B_0)=(88,705,913),
\]

which recovers \eqref{eq:seedidentity}.  The value of \(m_1\) in the
corollary shows that the very next member of the same Pell orbit is
already extraordinarily large; its remaining parameters are

\[
A_1=2884716843021136441536,\qquad
B_1=3734374369055532795889.
\]

The seed was found by a direct computer search, but the next member of its
Pell orbit is already far beyond the range of such a search.  The Pell
equation is useful because it turns one isolated example into an explicit
infinite family.

However, this is not the next ordered cubic run overall.  The generalized Pell equation \eqref{eq:generalizedpell} has
solutions belonging to other orbits.  In fact, there is an ordered solution
with the much smaller value
\[
 m=12005615252944088.
\]
More precisely,
\[
 (z,m)=(682091780239355,12005615252944088)
\]
gives
\[
 A=96105931914993680,\qquad
 B=124412740794926913,
\]
and
\[
 B-(A+2m-1)=4295578374045058>0.
\]
So this is also an ordered cubic run, and its value of $m$ is much smaller
than $360360696170473731112$.  So \(m_1\) is merely the next term in the orbit that begins at \(88\); 
smaller ordered solutions exist in other orbits of \eqref{eq:generalizedpell}.
\end{remark}

\section{The smaller identity}

The smaller identity \eqref{eq:small} also has a Pell reduction.  Since \((X,Y)=(57,36)\), impose
\[
 X=19z,\qquad Y=12z.
\]
Then \eqref{eq:centered} reduces to
\begin{equation}\label{eq:smallpell}
 53z^2-28m^2=29,
\end{equation}
with the seed \((z,m)=(3,4)\).  The associated Pell equation is
\[
 u^2-1484v^2=1,
\]
whose least positive solution is \((u,v)=(1695,44)\).  Multiplication
by this Pell unit gives the recurrence
\[
 z'=1695z+1232m,\qquad
 m'=2332z+1695m,
\]
which preserves \eqref{eq:smallpell}.  Starting from \((3,4)\), the
next solution is
\[
 (z,m)=(10013,13776).
\]
For \(X=19z\), \(Y=12z\), the corresponding
\[
 A=\frac{19z-2m+1}{2},\qquad B=12z-m+1
\]
satisfy
\[
 B-(A+2m-1)=-2518<0.
\]
So this Pell orbit gives more solutions of the cubic power-sum equation,
but the ordering condition fails already at the next step.
For example, the next iterate gives
\[
(m,A,B)=(13776,81348,106381),
\]
and hence the cubic identity
\[
\begin{aligned}
81348^3+81349^3+\cdots+108899^3
&=106381^3+106383^3+\cdots+133931^3\\
&=24,212,021,903,066,039,616.
\end{aligned}
\]
This is a genuine equality of cubic power sums, but it is not an ordered
cubic run: the right side starts at \(106381\), while the left side
continues up to \(108899\).  So the ratio \(317:200\) works better for
the ordered-run problem.

The cubic version behaves quite differently from the classical Pythagorean runs. We do not get a solution for every prescribed
length.  Instead, after centering the progressions and fixing a suitable
ratio of their centers, the problem leads to a generalized Pell equation.  The seed \( (m,A,B)=(88,705,913) \) lies on a Pell orbit
that yields infinitely many ordered identities in which \(2m\) consecutive
cubes equal \(m\) later cubes with common difference \(2\).  The
corresponding half-lengths themselves satisfy a second-order linear
recurrence.  An elementary question about consecutive cubes ends up producing a Pell equation and a linear recurrence.

The square identity \eqref{eq:square-run} and the cubic family
\eqref{eq:family} have rather different arithmetic characters.
For squares, the relevant difference factors linearly in \(A\), giving
the elementary formula \(A=K(2K+1)\) for every \(K\).  For the cubic
step-\(2\) variant, fixing a suitable ratio of progression centers
leads instead to a Pell-type equation, equivalently to a generalized Pell equation after normalization.

This problem is also part of the more general study of Diophantine
equations involving power sums of arithmetic progressions.  The closest
general reference among those cited here is \cite{BazsoEtAl2012}, which
studies equal values of two such power sums when the progression
parameters and exponents are fixed.  By contrast, in \eqref{eq:main}
the starting points and the common length parameter vary together, and
the result here is an explicit infinite family satisfying the additional
ordering condition.  The occurrence of a generalized Pell equation is
not itself novel; related Pell reductions for sums of consecutive cubes
occur, for example, in \cite{Pletser}.

Several questions remain even if we keep the same basic setup.  For example:
\begin{enumerate}
 \item Classify the positive integer solutions of \eqref{eq:main}.
 \item Determine which rational center ratios \(X:Y\) reduce
       \eqref{eq:centered} to generalized Pell equations possessing
       infinitely many positive solutions.
 \item Determine which such Pell families preserve the ordering
       \(B>A+2m-1\).
 \item Replace the right-hand common difference \(2\) by a general
       integer \(d\ge2\), or replace cubes by higher powers.
\end{enumerate}

\section*{Acknowledgments}

The computational verification and numerical experiments in this paper were performed using MATLAB. The author used OpenAI ChatGPT (GPT-5.6 Sol) during the preparation of this paper to assist with numerical searches, checking calculations, language corrections, and formatting. The author independently reviewed and verified the mathematical arguments, computations, references, originality, and accuracy of the paper and takes full responsibility for its entire content.

\end{document}